\documentclass[reqno]{amsart}
\usepackage[sorting=none, sorting=nyt, maxbibnames=99, backend=biber]{biblatex}
\bibliography{references.bib}

\usepackage{amssymb}
\usepackage{calrsfs}
\usepackage{graphics,graphicx, mathrsfs}
\usepackage{enumerate}
\usepackage{enumitem}
\usepackage{url}
\usepackage{xcolor}
\definecolor{vio}{rgb}{0.54, 0.17, 0.89}
\usepackage[ruled, lined, linesnumbered, longend]{algorithm2e}
\usepackage{hyperref}
\usepackage{titlesec}
\newtheorem{theorem}{Theorem}[section]
\newtheorem{lemma}[theorem]{Lemma}

\newtheorem{corollary}[theorem]{Corollary}
\numberwithin{equation}{section}

\theoremstyle{remark}
\newtheorem{remark}{Remark}

\titleformat{\section}
  {\normalfont\large\bfseries\centering}{\thesection}{1em}{}

\titleformat{\subsection}
  {\normalfont\bfseries}{\thesubsection}{1em}{}

\newcommand{\floor}[1]{\left\lfloor#1\right\rfloor}

\def\reals{\hbox{\rm I\kern-.18em R}}
\def\complexes{\hbox{\rm C\kern-.43em
\vrule depth 0ex height 1.4ex width .05em\kern.41em}}
\def\field{\hbox{\rm I\kern-.18em F}} 

\newcommand\blfootnote[1]{%
  \begingroup
  \renewcommand\thefootnote{}\footnote{#1}%
  \addtocounter{footnote}{-1}%
  \endgroup
}

\newenvironment{section*}[2][A]{
  \section*{#2}
  \renewcommand\thesection{#1}
  \setcounter{theorem}{0}}{}

\allowdisplaybreaks

\begin{document}

\title[Divisor Problem]{On the Generalised Divisor Problem}

\author{Sebastian Tudzi}
\address{School of Science, UNSW Canberra, Australia}
\email{s.tudzi@unsw.edu.au}
\date\today
\keywords{Divisor function $\cdot$ Class numbers $\cdot$ Explicit results
}

\begin{abstract}
In this paper, we apply the Dirichlet convolution method to
\begin{equation*}
    T_{k}(x)=\sum_{n \leq x} d_{k}(n),
\end{equation*}
for $k\ge 3$, where $d_{k}(n)$ is the number of ways to represent $n$ as a product of $k$ positive integer factors. We prove that for $k=3$, the error term $|\Delta_3(x)| <2.968x^{2/3}\log^{1/3}x$ for all $x\ge 2$. This improves 
the best-known explicit result established by Bordell{\`e}s for all $x\ge 2$. We extend this for all $k>3$ and obtain an explicit error term of the form $\Delta_{k}(x)=O\left(x^{\frac{k-1}{k}}(\log x)^{\frac{(k-1)(k-2)}{2k}}\right)$.  
\end{abstract}

\maketitle
\blfootnote{\textit{Affiliation}: School of Science, The University of New South Wales Canberra, Australia}
\blfootnote{\textit{Author}: Sebastian Tudzi (s.tudzi@unsw.edu.au)}
\blfootnote{\textit{Key phrases}: Piltz divisor function, Dirichlet convolution, explicit results}
\blfootnote{\textit{MSC classes}: 11N99}

\section{Introduction}
Let $d(n)$ be the number of positive integer divisors of $n$, or equivalently the number of ways to write $n$ as a product of two factors. Also, let $d_k(n)$ represent the number of ways to write $n$ as a product of $k$ positive integers. The study of the divisor problem of Dirichlet involves determining the asymptotic behaviour of 
\begin{equation*}
  T(x):=\sum_{n \leq x} d(n)  
\end{equation*}
as $x\rightarrow\infty$.
 Dirichlet \cite{dirichlet1851bestimmung} proved that 
\begin{equation}\label{meq}
    T(x)=x\left(\log x+2\gamma-1\right)+\Delta(x),
\end{equation}
where $\gamma$ is  the Euler--Mascheroni constant and the error term, $\Delta(x)= O(x^{1/2})$. For $\Delta(x)=O(x^{\alpha+\epsilon})$, where $\epsilon>0$, Berndt, Kim and Zaharescu \cite{MR3756337} provide a record of numerical values of $\alpha$ achieved over the previous century or more. The best known value for $\alpha$ as established by Huxley \cite{MR2005876} is $131/416$.   However, obtaining explicit constants for some of these better results of $\alpha$ remains a challenging task due to the complexity of the proof. Consequently, asymptotically weaker but explicit results continue to be valuable for applications.

 Berkane, Bordellès, and Ramaré \cite[Theorem 1.1]{MR2869048} provided fully explicit bounds for the error terms $\Delta(x)$ in \eqref{meq}. Their results for different ranges of $x$ are summarized as follows:
 \begin{align}
     |\Delta(x)|&\leq 0.961x^{1/2},\ \ x\geq 1,\label{eddy}\\
     |\Delta(x)|&\leq 0.482x^{1/2},\ \ x\geq 1981,\label{cele}\\
     |\Delta(x)|&\leq 0.397x^{1/2},\ \ x\geq 5560.\label{tk}
     \end{align}
    A much stronger explicit upper bound was provided in \cite[Theorem 1.2]{MR2869048}, featuring a significantly lower exponent of 
$x$ that holds for sufficiently large 
$x$. Specifically,
     \begin{align}
     |\Delta(x)|&\leq 0.764x^{1/3}\log x,\ \ x\geq 9995\label{fx}.
 \end{align} 
This result was later confirmed to hold for all $x\ge 5$ in \cite[footnote, p.~8]{Simoni2021AtkinsonsFF}. The bound in \eqref{fx} is consistent with Vorono{\"i}'s result \cite{MR1580627}, $\Delta(x)=O(x^{\frac{1}{3}}\log x)$,  and is currently the best known explicit result for $x\ge 6.86\times 10^{9}$.
 
    In \cite[Lemma 3.2]{MR2257286}, Bordell{\`e}s derived an asymptotic formula for 
\begin{equation*}
    T_3(x)=\sum_{n \leq x} d_{3}(n).
\end{equation*}
 Using Dirichlet convolution, he showed that  
\begin{equation}\label{fbk}
    T_3(x) = x\left(\frac{\log^2 x}{2} + (3\gamma - 1)\log x + 3\gamma^2 - 3\gamma - 3\gamma_1 + 1\right) + \Delta_3(x),
\end{equation}
where the error term for $x>670$ satisfies the bound    
\begin{equation}\label{fn}
    |\Delta_3(x)| \leq 2.36x^{2/3}\log x.
\end{equation}
Note that for a non-negative integer $r$, we define the Euler--Stieltjes constants as
\begin{equation}\label{fk}
\gamma_{r}:=\lim_{x\rightarrow\infty}\left(\sum_{n\le x}\frac{\log^{r}n}{n}-\frac{\log^{r+1}x}{r+1}\right),
\end{equation}
where $\gamma_{0}=\gamma$.

Building on this, Cully-Hugill and Trudgian \cite[Theorem 1]{MR4311680} introduced an alternative approach to derive an explicit error term for $T_4(x)$. Applying a similar convolution argument and the identity $d_4(n)=(d*d)(n)$ rather than $d_4(n)=(1*d_3)(n)$, they expressed $T_4(x)$ as:
    \begin{equation*}
        T_4(x)= 2\sum_{a\leq x^{1/2}}d(a)\sum_{b\leq x/a}d(b)-\left(\sum_{b\leq x^{1/2}}d(b)\right)^2.
         \end{equation*}
    From this, they derived the asymptotic formula: 
    \begin{align}\label{cvf}
       T_4(x)&=x\left(\frac{1}{6}\log^3x+\left(2\gamma-\frac{1}{2}\right)\log^2 x+(6\gamma^2-4\gamma-4\gamma_1+1)\log x\right)\notag 
   \\&\quad +(4\gamma^3-6\gamma^2+4\gamma-12\gamma\gamma_1+4\gamma_1+2\gamma_2-1)x+\Delta_4(x),   
   \end{align}
   where, for $x\geq 2$, the error term satisfies
    \begin{equation}\label{nana}
        |\Delta_4(x)|\leq 4.48x^{3/4}\log x.
    \end{equation}
    
 More generally, for $k\ge 2$, we define
\begin{equation*}
    T_k(x):=\sum_{n\leq x}d_k(n).
\end{equation*}
 Note that $T(x)$ and $d(n)$ corresponds to the case $k=2$. The problem of estimating $\Delta(x)$ is commonly known as the Dirichlet divisor problem. On the other hand, estimating $\Delta_k(x)$ for all $k>2$ is referred to as the general (or generalised) divisor problem, also known as the Piltz divisor problem. For positive integers $m_1,m_2,\hdots,m_k$, $T_k(x)$ is the number of positive integer coordinates beneath the hyperbolic surface  of the inequality $m_1m_2\hdots m_k\leq x$.
We write 
\begin{equation}\label{mjj}
    T_k(x)=xP_k(\log x)+\Delta_k(x),
\end{equation}
where $P_k(\log x)$ is a polynomial of degree of at most $k-1$ with a leading coefficient of $1/(k-1)!$ and $\Delta_k(x)=O(x^{\alpha_k+\epsilon})$ with $\alpha_k<1$ and $\epsilon>0$. In general, the divisor problem is the task of determining the optimal $\alpha_k$. A thorough discussion on $\Delta_k(x)$ is provided in \cite[Chapter 13]{MR1994094} (see also \cite[Chapter 12]{MR882550}). 

Following the identity $d_{k}(n) = (1 * d_{k-1})(n)$, one may apply the hyperbola
method to obtain a decomposition of the form
    \begin{equation*}
       T_k(x)= \sum_{a\leq x^{1/k}}\sum_{b\leq x/a}d_{k-1}(b)+\sum_{x^{1/k}< a\leq x}\sum_{b\leq x^{1/k}}d_{k-1}(b).
    \end{equation*}
As shown in \cite[Equation 12.1.4]{MR882550}, this gives the error term  
    \begin{equation}\label{ajua}
        \Delta_k(x)=O(x^{\frac{k-1}{k}}\log^{k-2} x).
    \end{equation}
   It is evident that \eqref{eddy}, \eqref{cele}, \eqref{tk}, and \eqref{fn} align well with general error term, \eqref{ajua}.  Compared with \eqref{ajua}, the result from \eqref{nana} shows a considerable reduction in the logarithmic exponent. This result not only improves upon the logarithmic exponent in the error term for $T_4(x)$ but also suggests the potential for similar refinements in the general case for all $k>2$.
   
For $k>2$, there are significant improvements in \eqref{ajua}, particularly the power of $x$ (see  \cite[Theorems 12.2 and 12.3]{MR882550}). Obtaining explicit versions of these bounds appears to be just as difficult as it does for the bounds for $\Delta(x)$. For this reason, we opted to make the bound for \eqref{ajua}, which is based on a more accessible hyperbola method, fully explicit along with a slight improvement in the logarithmic power.

Bounds on class number can be obtained using bounds on $T_k(x)$. Lenstra \cite[Section 6]{MR1129315} and Bordellès \cite[Lemma 1.1]{MR2257286} proved that if $h_{\mathbb{K}}$ represents the class number of a number field $\mathbb{K}$, then 
\begin{equation*}
h_{\mathbb{K}}\leq \sum_{m \leq b} d_{n_{\mathbb{K}}}(m), 
\end{equation*} 
where $n_{\mathbb{K}}$ is the degree of $\mathbb{K}$ and $b$ is the Minkowski bound of $\mathbb{K}$ (see also \cite[pp. 143]{MR4311680}). It is worth noting that the bound on $T_{k}(x)$ of the form 
\begin{equation}\label{vvp} 
T_k(x)\leq \frac{x\log^{k-1}x}{(k-1)!}(1+o(1)),
\end{equation} 
which has the same leading term as the polynomial in \eqref{mjj}, can be applied to obtain bounds $h_{\mathbb{K}}$. 

Several explicit results have been established for $h_{\mathbb{K}}$ (see, e.g., \cite{MR1917797}, \cite{MR2257286}, \cite{MR4311680}). However, these results have shortcomings since weaker versions of bounds on $T_{k}(x)$ are applied to obtain bounds on $h_{\mathbb{K}}$. For instance, the bound on $T_{k}(x)$ in \cite[Lemma 1.1]{MR1917797} and \cite[Theorem 2.1]{MR2257286} does not include the entire main term for $T_{k}(x)$. For $k=4$, \cite[Theorem 1]{MR4311680} provides the whole main term and an explicit error term, as seen in \eqref{cvf} and \eqref{nana}, respectively. Our goal for this paper is to present both the complete main term and an explicit error term for $T_k(x)$.

 The explicit bound in \eqref{vvp} highlights the importance of obtaining a general result for $T_k(x)$ with an improved error term, $\Delta_{k}(x)$, as they provide more precise and practical results for diverse applications. Inspired by these methods, the purpose of this paper is to obtain the required main term in \eqref{fbk} and an improved version of \eqref{fn} which features a reduced exponent of the $\log x$ factor. Additionally, we extend this method to derive an explicit upper bound estimate for the general error term in \eqref{ajua} with an improvement in the $\log x$ factor.
 
 In our first and second results, we  provide an estimate for the case $k=3$ and the general case $k\ge 3$. We show $k=3$ in order to get a better result that the general technique is unable to offer.
 \begin{theorem}\label{ric}
      For any real $x\geq 2$, we have
\begin{equation*}
T_3(x)=x\left(\frac{\log^2 x}{2}+(3\gamma -1)\log x +3\gamma^2-3\gamma-3\gamma_1+1\right)+\Delta_3(x),
    \end{equation*}
    where the error term $\Delta(x)$ satisfies the bound
    \begin{equation}\label{ss}
|\Delta_3(x)|< 2.968x^{2/3}\log^{1/3}x.
    \end{equation}
 \end{theorem}
 \begin{theorem}\label{dfg}
Let $k\ge 4$. There exist a constant $\lambda_{k}$ and a threshold $x_{0}=x_{0}(k)$ such that for all $x\ge x_{0}$, we have 
\begin{equation}\label{rtp}
    |\Delta_k(x)|<\lambda_{k}x^{\frac{k-1}{k}}(\log x)^{\frac{(k-1)(k-2)}{2k}},
\end{equation}
 where $\Delta_{k}(x)$ is defined by $\eqref{mjj}$. For $k=4,\,5,\,6$, the values of $\lambda_{k}$ and $x_{0}$ are given in Table~\ref{lambda}.
\begin{table}[!ht]
    \centering
    \begin{tabular}{|c|c|c|c|}
        \hline$k$  &$x_{0}$ & $c$&$\lambda_{k}$\\
        \hline
        $ 4 $\,& \,$10^{11}$& \ $0.449$ & \, $20.087$ \\
        $ 5 $\, & \,$10^{12}$& \ $0.105$ &\,  $114.332$\\
        $ 6 $\, & \,$10^{13}$& \ $0.024$ &\,  $646.358$\\
        \hline
    \end{tabular}
    \caption{Values of $k$, $c$, and $\lambda_{k}$ for $x \ge x_{0}$.}
    \label{lambda}
\end{table}
\vspace{-0.3em}
\par

Moreover, the values of $\lambda_{k}$ satisfy the recurrence relation
\begin{equation}\label{cxz}
    \lambda_{k}=\left((c^{\frac{1}{k-1}}+2)k+\frac{1}{c^{\frac{k-2}{k-1}}}\right)\lambda_{k-1}+\frac{k-1}{8c^2}[(k+5)c^{3}+12c+k-2],
\end{equation}
 where the constant $c$ is given by
 \begin{equation}\label{ttr}
        c=\left(\frac{3(k-1)}{2k\lambda_{k-1}}\right)^{\frac{k-1}{k}}.
    \end{equation}
    The initial value of $\lambda_{3}$ can be taken to be $3.281$. Table~\ref{lew} presents some computed values of $c$ for the corresponding values of $k$ and $\lambda_{k-1}$.
    \begin{table}[!ht]
    \begin{tabular}{|c|c|c|}
    \hline
        $k$ &$\lambda_{k-1}$ & $c$\\
        \hline
        $ 3 $\, & \, $0.397$ \, & \,$1.852$ \\
        $ 4 $\, &\,  $3.281$ \,  & \, $0.449$  \\ $ 5 $\,  & \,$20.087$\,& \,$0.105$\\
        $ 6 $\,  & \,$114.332$\,& \,$0.024$\\ \hline
    \end{tabular}
    \caption{Values of $k$, $\lambda_{k-1}$, and $c$.}
    \label{lew}
\end{table}
\vspace{-0.3em}
\par
\end{theorem}

\begin{remark}\label{dmf}
  The proof of Theorem \ref{dfg} for $k=3$ yields $|\Delta_{3}(x)|<3.275x^{2/3}\log^{1/3}x$ for all $x\ge 1.704\cdot10^{10}$ if we take $\lambda_{2}=0.397$. This result indicates a somewhat worse explicit constant compared with \eqref{ss}, necessitating the separation of $k=3$. Also, for $k=4$, our result only improves on \eqref{nana} for $x_{0}\ge 3.330\times 10^{175}$. 
 \end{remark}
    \begin{corollary}
  For $k\ge 7$, we have 
  \begin{equation*}
      \lambda_{k}\le 5^{k-3}(k!/3!)^{2}\lambda_{3}.
  \end{equation*}
\end{corollary}
This is a direct consequence  of substituting \eqref{ttr} into \eqref{cxz} and applying the estimates $\lambda_{k-1}\ge \lambda_3>1$ and $0<c<1$ for all $k\ge 7$, along with an iteration of the result from $k=4$.

We provide the proof of Theorem \ref{ric} in Section 2, whereas in Section 3 we present the proof of Theorem \ref{dfg}.

\section{Proof of Theorem \ref{ric}}
The proof is essentially a simple modification of \cite[Lemma 3.2]{MR2257286}. Using the hyperbola method and the identity $d_3(n) = (1*d)(n)$, we have
\begin{align*}
     T_3(x)&=\sum_{n\leq x}(1*d)(n)\\
&=\sum_{a\le u}\sum_{b\leq\frac{x}{a}}d(b)+\sum_{b\le v}\sum_{a\leq\frac{x}{b}}d(b)-\sum_{a\le u}\sum_{b\le v}d(b)\\
&=\sum_{a\le u}\sum_{b\leq\frac{x}{a}}d(b)+\sum_{b\le v}d(b)\sum_{a\leq\frac{x}{b}}1-\sum_{a\le u}\sum_{b\le v}d(b)
\end{align*}
for any $u, v\ge 1$ such that $uv=x$. By applying \eqref{meq}, we obtain 
\begin{align}\label{kwe}
   T_3(x)&=x\sum_{a\le u}\frac{\log (x/a)}{a}+(2\gamma-1)x\sum_{a\le u}\frac{1}{a}+\sum_{a\le u}\Delta(x/a) 
   + x\sum_{b\le v}\frac{d(b)}{b}-u\sum_{b\le v}d(b)\notag\\&\quad + E_2(v),
   \end{align}
   where $E_2(v)=\vartheta\left(\sum_{b\le v}d(b)\right)$. Note that $\vartheta$ refers to a number with an absolute value of no more than one. To obtain an explicit upper bound for $E_{2}(v)$, when $v \ge 6\times 10^{5}$ we follow
the approach of \cite[Section~6]{MR2869048}, where the error term
\[
\Delta(x)=\sum_{n\le x} d(n)-x(\log x+2\gamma-1)
\]
is computed for large values of $x$, and the size of the
remainder is explicitly controlled.  Applying the same argument to
$\sum_{b\le v} d(b)$, we obtain
\[
\left|
\sum_{b\le v} d(b) - \bigl(v\log v + (2\gamma - 1)v\bigr)
\right|
\le 0.173\, \sqrt{v}
\]
for all $v\ge 6\times 10^{5}$.
Hence, we have
\begin{equation}\label{amo}
    |E_{2}(v)|
    \le v\log v + (2\gamma - 1)v + 0.173\,\sqrt{v}.
\end{equation}

   From \cite[Lemma 3.3]{MR2257286}, we have the following result for $e^{3/2}\leq u\leq x$:
   \begin{equation}\label{bea}
       \sum_{a\le u}\frac{\log (x/a)}{a}=\log x\log u-\frac{\log^2 u}{2}+\gamma\log x-\gamma_1 +E_3(x,u),
   \end{equation}
   where the bound on remainder term is given by 
   \begin{equation}\label{cel}
       |E_3(x,u)|\leq  \frac{\log (x/u)}{2u}+\frac{\log x}{4u^2}.
   \end{equation}
Also, applying the Euler--Maclaurin summation formula yields
\begin{equation}\label{eco}
    \sum_{a\le u}\frac{1}{a}=\log u+\gamma+E_4(u),
\end{equation}
where $E_4(u)=\vartheta(c/u)$ and $c$ is a constant. In particular, 
\begin{align}\label{ama}
    |E_4(u)|&\le \frac{1}{2u}+\frac{1}{12u^{2}}\notag\\
    &\le \frac{0.501}{u}
\end{align}
for all $u\ge 84$.
 A result from Riesel and Vaughan \cite[Lemma 1]{MR706639} showed  that for $v> 1$, 
\begin{equation}\label{pet}
    \sum_{b\le v}\frac{d(b)}{b}=\frac{\log^2 v}{2}+2\gamma\log v+\gamma^2-2\gamma_1+E_5(v),
\end{equation}
where $|E_5(v)|<1.641v^{-1/3}$. This error term was improved by Berkane et al.\ \cite[Corollary 2.2]{MR2869048} to $|E_5(v)|<1.16v^{-1/3}$. A further improvement was established by Platt and Trudgian \cite[Lemma 1]{MR4100646} in which for $v\geq 1$, $|E_5(v)|<0.6877v^{-2/5}$. However, to obtain a bound of the form $x^{2/3}\log^{1/3}x$, we require $|E_5(v)|\leq cv^{-1/2}$, where $c$ is a constant. According to Cully-Hugill and Trudgian \cite[Section 2]{MR4311680}, for $v\geq 6\times10^{5}$, 
\begin{equation}\label{debb} 
|E_5(v)|\leq 1.001v^{-1/2}.
\end{equation}
This bound is suitable for our purposes.

Again, the fifth sum in the right-hand side of \eqref{kwe} is given by
\begin{equation}\label{mv}
    \sum_{b\leq v}d(b)=v\log v+(2\gamma -1)v+E_6(v)
\end{equation}
for $v\geq 6\times 10^{5}$ where,  
\begin{equation}\label{jul}
    |E_6(v)|\leq 0.173\sqrt{v}.
\end{equation}

Finally, let
\begin{equation}
    E_7(x,u)=\sum_{a\leq u}\Delta(x/a)\label{kaf}.
\end{equation}
 Note that throughout this part of the proof, we assume $x\ge 1.1 \cdot 10^{10}$. The remaining range can be verified numerically.
 To obtain a bound on \eqref{kaf}, we take 
 \begin{equation*}
 u\asymp x^{1/3}(\log x)^{2/3}.    
 \end{equation*}
 Applying the estimate in \eqref{tk}, and noting that $x/a\ge x/u>5560$, we obtain 
\begin{align}    
    |E_7(x,u)|&\leq\sum_{a\leq u}|\Delta(x/a)|\notag\\
    & \le 0.397\sqrt{x} \sum_{a\le u} \frac{1}{\sqrt{a}} \notag\\
    &\le 0.397\sqrt{x}\left(1+\int_{1}^{u}\frac{1}{\sqrt{a}}\,\text{d}a\right) \notag\\
    &<0.794\sqrt{xu}\label{syl}.
\end{align}

Substituting \eqref{bea}, \eqref{eco}, \eqref{pet}, \eqref{mv}, \eqref{kaf} and  $v=x/u$ into \eqref{kwe} and simplifying gives
\begin{align*}
    T_3(x)&=x\left(\frac{\log^2 x}{2}+(3\gamma-1)\log x+3\gamma^2-3\gamma_1-3\gamma+1\right)+\Delta_{3}(x, u),
\end{align*}
where 
\begin{align}\label{err}
   \Delta_{3}(x, u)&=E_2(x, u)+xE_3(x, u)+x(2\gamma-1)E_4(u)+xE_5(x, u)-uE_6(x, u)\notag\\&\quad+E_7(x, u).
\end{align}
 Bounding \eqref{err} and substituting \eqref{amo}, \eqref{cel}, \eqref{ama}, \eqref{debb}, \eqref{jul} and \eqref{syl}  gives
     \begin{align}\label{abn}
         |\Delta_{3}(x, u)|&< \frac{3x\log (x/u)}{2u}+\frac{x\log x}{4u^2}+\frac{1.501(2\gamma-1)x}{u}+1.968\sqrt{xu}+0.173\sqrt{x/u}.
     \end{align}
     
     Since $u\geq 1 $ and $x\geq u$, we focus on identifying the dominant terms on the right of \eqref{abn}. Notably, the first term and the fourth term grow faster than the other terms as $x$ increases. Therefore, we set 
     \begin{equation}\label{wer}
         u=Ax^{1/3}\log^{2/3}x,
     \end{equation}
     with $A>0$, and optimize over $A$. 

With this choice, the main contribution to the coefficient of $x^{2/3}\log^{1/3}x$ is
\begin{equation*}
    F(A)\;=\;\frac{1}{A} + 1.968\sqrt{A}.
\end{equation*}
Minimizing \(F(A)\) over \(A>0\) gives $A_0\approx1.011$, and the corresponding minimal value $F(A_0)=2.968$. Hence, the simple choice $A=1$ is essentially optimal and yields the same constant as $F(A_{0})$ with $A_{0}=1.011$.
     
     Substituting \eqref{wer} with $A=1$ into \eqref{abn} and simplifying gives
     \begin{align}\label{fbfi}
    |\Delta_3(x)|
&<2.968x^{2/3}\log^{1/3}x+\frac{(1.501(2\lambda-1)-\log\log x)x^{2/3}}{\log^{2/3}x}+\frac{0.423x^{1/3}}{\log^{1/3}x}.
     \end{align}
We observe that the sum of the last two terms in \eqref{fbfi} is negative for all $x\ge 4$.
However, in order for \eqref{abn} to apply,
the parameter $u$ must satisfy
\begin{equation*}
    84 \le x^{1/3}(\log x)^{2/3} 
    \le \frac{x}{600000}.
\end{equation*}
A numerical check shows that the lower inequality holds for all $x\ge 7454.828$ and the upper inequality holds when $x \ge 1.1\times 10^{10}$. 

Therefore, we obtain 
\begin{equation*}
|\Delta_3(x)|< 2.968\,x^{2/3}\log^{1/3} x.
\end{equation*}  
for all $x\ge 1.1\cdot 10^{10}$.
     We  confirm that this result is true for all $2 \le x < 1.1\cdot 10^{10}$ by computing the partial sums of $d_3(n)$ using \textit{Mathematica}.

\section{Useful lemmas for the proof of Theorem \ref{dfg}}
In this section, we prove the main lemmas essential for the proof of Theorem \ref{dfg}. 
For our purposes, we assume that the well-known result for $k=2$ is as follows:
            \begin{equation*}
                T_2(x)=xP_2(\log x)+\Delta(x),
            \end{equation*}
            where the error term,  $|\Delta(x)|\le\lambda x^\frac{1}{2}$ with $\lambda$ as established in \eqref{eddy}, \eqref{cele} and \eqref{tk}.
We also assume that
\begin{equation}\label{mlo}
             T_{k-1}(x)=xP_{k-1}(\log x)+\Delta_{k-1}(x),
    \end{equation}
            where the error term
\begin{equation}\label{bnj}   |\Delta_{k-1}(x)|\leq \lambda_{k-1}x^{\frac{k-2}{k-1}}(\log x)^{\frac{(k-2)(k-3)}{2(k-1)}} 
            \end{equation}
is true for all positive integers $k>2$.
\begin{lemma}\label{gvg}
    Let $k\ge 3$ and $N\ge 1$ be a natural number. Then we have
    \begin{equation}\label{kawa}
        \sum_{n\le x/U}\frac{d_{k-1}(n)}{n}=\sum_{n\le N}\frac{T_{k-1}(n)}{n(n+1)}+\frac{T_{k-1}(N)}{N+1}.
    \end{equation}
\end{lemma}
\begin{proof}
 We begin with the arithmetic function $d_{k-1}(n)=T_{k-1}(n)-T_{k-1}(n-1)$.  For $N=\floor{x/U}$, the sum on the left-hand side of \eqref{kawa} is given by:
        \begin{align*}
        \sum_{n\le N}\frac{d_{k-1}(n)}{n}&=\sum_{n\le N}\frac{T_{k-1}(n)}{n}-\sum_{n\le N}\frac{T_{k-1}(n-1)}{n}.
        \end{align*}
Shifting the index in the second sum on the right and noting that $T_{k-1}(0)=0$, the right-hand side becomes
\begin{align*}
\sum_{n\le N}\frac{T_{k-1}(n)}{n}-\sum_{n\le N-1}\frac{T_{k-1}(n)}{n+1}=\sum_{n\le N-1}\frac{T_{k-1}(n)}{n(n+1)}+\frac{T_{k-1}(N)}{N}.
        \end{align*}
        Adjusting the summation limit to include $N$, we add and subtract the term $\frac{T_{k-1}(N)}{N(N+1)}$. This gives \eqref{kawa} as required.
\end{proof}

Next, we offer various integral estimates required for the proof of Theorem \ref{dfg}. Although it is possible to obtain a more precise estimate for these integrals, the current estimate is sufficient for our purposes.
\begin{lemma}\label{hhn}
Let $1<a<x$, $\alpha>0$, $0<\beta<1$ and $U\ge 1$. Then we have
        \begin{align*}
\int_{1}^{U}\frac{\log^\alpha(x/a)}{a^{\beta}}\,\mathrm{d}a&\le \frac{\left(U^{1-\beta}-1\right)\log^{\alpha}x}{1-\beta}.
        \end{align*}
                \end{lemma}
                \begin{proof}
        Since $\alpha$ and $\beta$ are positive  and $1<a<x$, it follows immediately that $\log x>\log a$. Hence, we have 
        \begin{align*}
    \int_{1}^{U}\frac{\log^\alpha(x/a)}{a^{\beta}}\,\mathrm{d}a&\le\log^{\alpha}x\int_{1}^{U}\frac{1}{a^{\beta}}\,\mathrm{d}a.
        \end{align*} 
    Integrating and simplifying concludes the proof.
    \end{proof}
    
        \begin{lemma}\label{ghj}
        Given that $x/U>1$ and $k>2$, we have
            \begin{align*}
        \int_{x/U}^{\infty} \frac{(\log t)^\frac{(k-2)(k-3)}{2(k-1)}}{t^{\frac{1}{k-1}}(t+1)}\,\mathrm{d}t\le\frac{(k-1)U^{\frac{1}{k-1}}(\log (x/U))^\frac{(k-2)(k-3)}{2(k-1)}}{x^{\frac{1}{k-1}}}\left(1+\frac{k^2-3k+4}{2\log(x/U)}\right).
            \end{align*}
        \end{lemma}
        \begin{proof}
             Since $k>2$, the integral converges. Therefore, to estimate this integral, we use the substitution $t=e^m$. This yields
          \begin{align*}
          \int_{x/U}^{\infty} \frac{(\log t)^\frac{(k-2)(k-3)}{2(k-1)}}{t^{\frac{1}{k-1}}(t+1)}\,\mathrm{d}t&= \int_{\log(x/U)}^{\infty} \frac{m^\frac{(k-2)(k-3)}{2(k-1)}}{e^{\frac{m}{k-1}}(1+e^{-m})}\,\mathrm{d}m.
        \end{align*}
Since $m\ge \log(x/U)$ and $x>U>1$, it follows that $1\leq 1+e^{-m}< 2$. So we have  
\begin{align*}
    \int_{\log(x/U)}^{\infty} \frac{m^\frac{(k-2)(k-3)}{2(k-1)}}{e^{\frac{m}{k-1}}(1+e^{-m})}\,\mathrm{d}m&\leq \int_{\log(x/U)}^{\infty} m^\frac{(k-2)(k-3)}{2(k-1)}\cdot e^{\frac{-m}{k-1}}\,\mathrm{d}m.
    \end{align*}
    Using integration by parts gives 
    \begin{align*}
    &\leq\frac{(k-1)U^{\frac{1}{k-1}}(\log (x/U))^\frac{(k-2)(k-3)}{2(k-1)}}{x^{\frac{1}{k-1}}}\left(1+\frac{k^2-3k+4}{2\log(x/U)}\right).
\end{align*}
        \end{proof}
       In a final sequence of lemmas, we use the estimates from the previous lemmas.
\begin{lemma}\label{ccx}
Let $k\ge 3$ be an integer and let $x>U>1$. Then
\begin{equation}\label{nd}
    \sum_{a\le U} T_{k-1}\!\left(\frac{x}{a}\right)
    = M_{k}(x,U) + R_{2}(k,x,U),
\end{equation}
where
\[
M_{k}(x,U):= \sum_{r=0}^{k-2} Q_{\,k-r-2}(\log x)\left(\frac{\log^{\,r+1}U}{r+1}+\gamma_{r}\right),
\]
the functions $Q_{m}$ are polynomials in $\log x$ of degree at most $m$, and the constants
$\gamma_{r}$ is as defined in \eqref{fk}.
The error term satisfies
\begin{equation}\label{dds}
\begin{split}
    |R_{2}(k,x,U)|
    &\le \lambda_{k-1}(k-1)\,x^{\frac{k-2}{k-1}}\Big(U^{\frac{1}{k-1}}-1\Big)
        (\log x)^{\frac{(k-2)(k-3)}{2(k-1)}}\\
    &\qquad{}+(k-1)\log^{\,k-2}x\left(\frac{1}{2U}+\frac{k-2}{8U^{2}\log U}\right).
\end{split}
\end{equation}
\end{lemma}

\begin{proof} 
     Applying \eqref{mlo} gives
     \begin{equation}\label{aug}
        \sum_{a\le U}T_{k-1}\left(\frac{x}{a}\right)=x\sum_{a\le U}\frac{1}{a}P_{k-1}\left(\log \left(\frac{x}{a}\right)\right)+E(k,x,U),
            \end{equation}
    with an error term
\begin{equation}\label{dav}
        E(k,x,U)=\sum_{a\le U}\Delta_{k-1}\left(\frac{x}{a}\right).
            \end{equation}
             To obtain a bound for $E(k,x,U)$, we use the bound on $\Delta_{k-1}(x)$ in \eqref{bnj}. That is  
            \begin{align*}
         |E(k,x,U)|\leq\lambda_{k-1}x^{\frac{k-2}{k-1}}\sum_{a\leq U}\frac{(\log (x/a))^{\frac{(k-2)(k-3)}{2(k-1)}}}{a^{\frac{k-2}{k-1}}}.
                \end{align*}
        Using the result from Lemma \ref{hhn} gives
        \begin{align}\label{eerr}
        |E(k,x,U)|&<\lambda_{k-1}(k-1)\left(U^\frac{1}{k-1}-1\right)x^\frac{k-2}{k-1}(\log x)^{\frac{(k-2)(k-3)}{2(k-1)}}.
                \end{align}  
            
            The sum on the right-hand side of equation \eqref{aug} can be expressed as:
         \begin{align}\label{ghh}
         \sum_{a\le U}\frac{1}{a}P_{k-1}\left(\log \left(\frac{x}{a}\right)\right)&=\sum_{r=0}^{k-2}Q_{k-r-2}(\log x)\sum_{a\leq U}\frac{\log^r a}{a}.
         \end{align}
          Note that when $r=0$, $Q_{k-2}(\log x)=P_{k-1}(\log x)$. Applying the Euler--Maclaurin summation formula to the sum over $a\leq U$ in the right-hand side of \eqref{ghh} for a non-negative integer $r$ in the above equation yields   
    \begin{equation}\label{hhgf}
        \sum_{a\leq U}\frac{\log^{r}a}{a}=\frac{(\log U)^{r+1}}{r+1}+\gamma_{r}-B_{1}(U)f_{r}(U)+\frac{B_{2}(U)}{2}f_{r}'(U)+E_{r}(U),
    \end{equation}
    where $f_r(U)=(\log U)^{r}/U$, $B_1(U)=U-\floor{U}-\frac{1}{2}$, $B_{2}(U)=\frac{1}{2}\left(U-\floor{U}-\frac{1}{2}\right)^2$, and the error term 
    \begin{equation*}
        E_{r}(U)=\frac{1}{2}\int_{U}^{\infty}B_{2}(t)f_{r}''(t)\mathrm{d}t.
    \end{equation*} 
Hence, we have 
    \begin{align}\label{bmn}
               \sum_{r=0}^{k-2}Q_{k-r-2}(\log x)\sum_{a\leq U}\frac{\log^r a}{a}&=M_{k}(x,U)+E_{*}(k,x,U),
            \end{align}
            where the main term is given by 
            \begin{align*}
        M_{k}(x,U)= \sum_{r=0}^{k-2}Q_{k-r-2}(\log x)\left(\frac{\log^{r+1}U}{r+1}+\gamma_{r}\right)
            \end{align*}
            and the error term  
            \begin{align*}
                E_{*}(k,x,U)=\sum_{r=0}^{k-2}Q_{k-r-2}(\log x)\left(-B_{1}(U)f_{r}(U)+\frac{B_{2}(U)}{2}f_{r}'(U)+E_{r}(U)\right).
            \end{align*}
Since $-\frac{1}{2}\leq B_1(x)<\frac{1}{2}$ and $0\leq B_2(x)\leq\frac{1}{8}$, we have 
\begin{align*}
    |E_{*}(k,x,U)|&\leq \sum_{r=0}^{k-2}|Q_{k-r-2}(\log x)|\left|-B_{1}(U)f_{r}(U)+\frac{B_{2}(U)}{2}f_{r}'(U)+E_{r}(U)\right|.
    \end{align*}

To establish a bound for $Q_{k-r-2}(\log x)$, we begin by noting that
\begin{equation*}
   P_{k-1}(t)=\frac{t^{k-2}}{(k-2)!}+O(t^{k-3}). 
\end{equation*}
Now, by expanding $P_{k-1}(\log x-\log a)$ as a polynomial in $\log a$, we observe that the coefficient of
$\log^{r}a$ is
\begin{equation*}
    Q_{k-r-2}(\log x)
 = \frac{1}{(k-2)!}\binom{k-2}{r}(\log x)^{k-r-2}
   + O((\log x)^{k-r-3}).
\end{equation*}
For $0 \le r \le k-2$, the constant
\begin{equation*}
\frac{1}{(k-2)!}\binom{k-2}{r}
\end{equation*}
is bounded above by an absolute constant depending only on $k$, and in particular is at most $1$.
Since the corresponding error term is of lower order, the leading term dominates for sufficiently large $x$. Hence,  we have
\begin{equation*}
|Q_{k-r-2}(\log x)| \le (\log x)^{k-r-2}.    
\end{equation*}

It follows that 
    \begin{align*}
   |E_{*}(k,x,U)|&\leq\sum_{r=0}^{k-2}(\log x)^{k-r-2}\left(\frac{\log^r(U)}{2U}\right)\\&\quad+\sum_{r=0}^{k-2}(\log x)^{k-r-2}\left(\frac{\log^{r-1}(U)(r-\log U)}{16U^2}+C_r\right),
\end{align*}
where 
\begin{align}\label{dbbs}
    C_r=|E_r(U)|&\leq\left| \frac{1}{16}\int_{U}^{\infty}f_{r}''(t)\mathrm{d}t\right|\notag\\
    &=\left|\frac{f_{r}'(U)}{16}\right|\notag\\
    &\le \frac{\log^{r-1}U}{16U^{2}}(r+\log U).
\end{align}
As a result, 
\begin{align}\label{dcx}
|E_{*}(k,x,U)|&\le\sum_{r=0}^{k-2}(\log x)^{k-r-2}\left(\frac{\log^{r}U}{2U}+\frac{r\log^{r-1}U}{8U^{2}}\right)\notag\\
&\le (k-1)\log^{k-2} x\left(\frac{1}{2U}+\frac{k-2}{8U^{2}\log U}\right).
    \end{align} 
Substituting \eqref{dav} and \eqref{bmn} into \eqref{aug} gives \eqref{nd} as required, where $R_{2}(k,x,U)$ is the sum of $E(k,x,U)$ and $E_{*}(k,x,U)$. Also, summing \eqref{eerr} and \eqref{dcx} we obtain \eqref{dds} as required.
\end{proof}
\begin{lemma}\label{cmc}
For $k>2$ and $x>U>1$, we have
\begin{equation}\label{fes}
    \sum_{n\le x/U}\frac{d_{k-1}(n)}{n}=p_k(\log(x/U))+R_{3}(k,x,U),
\end{equation}
where $p_k(x)$ is a polynomial of degree at most $k$ and
\begin{align}\label{mma}
    |R_{3}(k,x,U)|&\le\frac{(k^{3}-4k^{2}+7k-4)\lambda_{k-1}U^{\frac{1}{k-1}}(\log (x/U))^\frac{k^{2}-7k+8}{2(k-1)}}{2x^{\frac{1}{k-1}}}\notag\\&\quad+\frac{k\lambda_{k-1}U^{\frac{1}{k-1}}(\log (x/U))^\frac{(k-2)(k-3)}{2(k-1)}}{x^{\frac{1}{k-1}}}+|\xi_{*}(k,x,U)|.
\end{align}
The term $\xi_{*}(k,x,U)$ will be estimated explicitly in Lemma~\ref{lem:Mstar}.

\end{lemma}
\begin{proof}
    To evaluate the sum on the left-hand side of \eqref{fes}, we use \eqref{mlo} with the result in Lemma \ref{gvg}. This allows us to derive the following expression:
           \begin{align}\label{bdn}
        \sum_{n\le x/U}\frac{d_{k-1}(n)}{n}&=\sum_{n\le x/U}\left(\frac{P_{k-1}(\log n)}{(n+1)}+\frac{\Delta_{k-1}(n)}{n(n+1)}\right) +\frac{NP_{k-1}(\log n)}{N+1}+\frac{\Delta_{k-1}(N)}{N+1}\notag\\
&=M_{*}(k,x,U)+\xi(k,x,U),
        \end{align}
    where the main term
    \begin{align}\label{ktu}
        M_{*}(k,x,U):=\sum_{n\le x/U}\frac{P_{k-1}(\log n)}{(n+1)}+\frac{NP_{k-1}(\log N)}{N+1},
    \end{align}
    and the error term
    \begin{align*}
\xi(k,x,U)&=\sum_{n\le x/U}\frac{\Delta_{k-1}(n)}{n(n+1)}+\frac{\Delta_{k-1}(N)}{N+1}\\
&=\sum_{n\ge 1}\frac{\Delta_{k-1}(n)}{n(n+1)}-\sum_{n> x/U}\frac{\Delta_{k-1}(n)}{n(n+1)}+\frac{\Delta_{k-1}(N)}{N+1}.
    \end{align*}
Taking the absolute value of both sides yields
\begin{align}\label{dx}
|\xi(k,x,U)|&\le C_{k}+\left|\sum_{n> x/U}\frac{\Delta_{k-1}(n)}{n(n+1)}\right|+\left|\frac{\Delta_{k-1}(N)}{N+1}\right|\notag\\
&\le C_{k}+\sum_{n> x/U}\frac{|\Delta_{k-1}(n)|}{n(n+1)}+\frac{|\Delta_{k-1}(N)|}{N+1},
\end{align}
where $C_{k}$ is a constant depending on $k$. 

To obtain an upper bound estimate for the sums on the right-hand side, we start with the inductive bound in \eqref{bnj}. That is,
        \begin{align}\label{edc}
            |\Delta_{k-1}(n)|\le \lambda_{k-1}n^{\frac{k-2}{k-1}}(\log n)^\frac{(k-2)(k-3)}{2(k-1)}.
        \end{align}
Now, using \eqref{edc} we bound the second term on the right-hand side of \eqref{dx} as follows: 
        \begin{align*}
         \sum_{n> x/U}\frac{|\Delta_{k-1}(n)|}{n(n+1)}&\leq \lambda_{k-1}\sum_{n> x/U}\frac{n^{\frac{k-2}{k-1}}(\log n)^\frac{(k-2)(k-3)}{2(k-1)}}{n(n+1)}\\
          &\leq \lambda_{k-1}\int_{x/U}^{\infty}\frac{t^{\frac{k-2}{k-1}}(\log t)^\frac{(k-2)(k-3)}{2(k-1)}}{t(t+1)}\,\mathrm{d}t\\
          & =\lambda_{k-1}\int_{x/U}^{\infty}\frac{(\log t)^\frac{(k-2)(k-3)}{2(k-1)}}{t^{\frac{1}{k-1}}(t+1)}\,\mathrm{d}t.
          \end{align*}
          The estimated result for this integral is immediate from Lemma \ref{ghj}. Therefore,
          \begin{align}\label{nmu}
             \sum_{n> x/U}\frac{|\Delta_{k-1}(n)|}{n(n+1)}&\leq \frac{(k-1)\lambda_{k-1}U^{\frac{1}{k-1}}(\log (x/U))^\frac{(k-2)(k-3)}{2(k-1)}}{x^{\frac{1}{k-1}}}\left(1+\frac{k^2-3k+4}{2\log(x/U)}\right).
          \end{align}  
        
        Finally, considering the last term of \eqref{dx}, since $N=\lfloor x/
        U\rfloor$, it follows that $N\leq x/U<N+1$. Therefore, it follows that the second term on the right of \eqref{dx} is 
        \begin{align}\label{hhs}
         \frac{|\Delta_{k-1}(N)|}{N+1}
         &\leq \frac{\lambda_{k-1}U^{\frac{1}{k-1}}(\log (x/U))^\frac{(k-2)(k-3)}{2(k-1)}}{x^{\frac{1}{k-1}}}.
        \end{align}
        Summing \eqref{nmu} and \eqref{hhs}, and substituting the result into \eqref{dx} gives:
        \begin{align}\label{bni}
            |\xi(k,x,U)|&\leq \frac{\lambda_{k-1}(k^{3}-4k^{2}+7k-4)U^{\frac{1}{k-1}}(\log (x/U))^\frac{k^{2}-7k+8}{2(k-1)}}{2x^{\frac{1}{k-1}}}\notag\\&\quad+\frac{k\lambda_{k-1}U^{\frac{1}{k-1}}(\log (x/U))^\frac{(k-2)(k-3)}{2(k-1)}}{x^{\frac{1}{k-1}}}.
        \end{align}
        
        Now, considering \eqref{ktu}, we obtain
        \begin{equation}\label{rup}
           M_{*}(k,x,U)=p_k(\log(x/U))+\xi_{*}(k,x,U),
        \end{equation}
        where $\xi_{*}(k,x,U)$ is the error term. 
        With $R_{3}(k,x,U)=\xi(k,x,U)+\xi_{*}(k,x,U)$, substituting \eqref{rup} into \eqref{bdn} yields \eqref{fes}. 
\end{proof}

\begin{lemma}\label{lem:Mstar}
Let $k \ge 3$, $x \ge 2$, $1 \le U \le x$ and $N = x/U$. Then
\[
M_{*}(k,x,U)= p_k(\log(x/U)) + \xi_{*}(k,x,U),
\]
where $M_{*}(k,x,U)$ is as defined in \eqref{ktu}, $p_k$ is a polynomial of degree $k-1$, and the error term satisfies
\[
|\xi_{*}(k,x,U)|\le\frac{(k-1)(k+5)U}{8x}(\log x)^{k-2}.
\] 
\end{lemma}

\begin{proof}
We begin by expressing the first term on the right-hand side of \eqref{ktu} as
\begin{equation*}
   \sum_{n \le N} \frac{P_{k-1}(\log n)}{n+1}\le \sum_{n \le N} \frac{P_{k-1}(\log n)}{n},
\end{equation*}
where the polynomial
\begin{equation*}
    P_{k-1}(\log n)=\sum_{j=0}^{k-2}a_{j}(\log n)^{k-j-2},
\end{equation*}
with $0<a_{j}\le 1$. Since $n \ge e$ implies that $(\log n)^{k-j-2}\le (\log n)^{k-2}$ for all
$0 \le j \le k-2$, and the coefficients $a_j$ are bounded, it follows that
\begin{equation}\label{ggbb}
    P_{k-1}(\log n)\le (k-1)(\log n)^{k-2}.
\end{equation}
Hence, we have
\[
\sum_{n \le N} \frac{P_{k-1}(\log n)}{n}
\le
(k-1)\sum_{n \le N} \frac{(\log n)^{k-2}}{n}.
\]

Applying the Euler--Maclaurin summation formula in \eqref{hhgf}, the sum on the right-hand side is given by
\begin{equation}
   \sum_{n \le N} \frac{(\log n)^{k-2}}{n}\le \frac{(\log N)^{k-1}}{k-1}+\gamma_{k-2}+\frac{f_{k-2}(N)}{2}+\frac{f_{k-2}'(N)}{16}+E_{k-2}(N).
\end{equation}
Hence, we have
\begin{equation*}
   p_{k}(\log N)= (\log N)^{k-1}+(k-1)\gamma_{k-2}
\end{equation*}
and the error term is given by
\begin{equation}\label{ddpp}
    \xi_{*}(k,N)=(k-1)\left(\frac{(\log N)^{k-2}}{2N}+\frac{(\log N)^{k-3}(k-\log N-2)}{16N^{2}}+E_{k-2}(N)\right),
\end{equation}
with
\begin{equation*}
    |E_{k-2}(N)|\le \frac{(\log N)^{k-3}}{16N^{2}}(k+\log N-2),
\end{equation*}
which follows by replacing \(r\) with \(k-2\) in \eqref{dbbs}. For $k\ge3$ and $N=x/U\ge e$, we bound \eqref{ddpp} as follows
\begin{equation}\label{jhf}
    |\xi_{*}(k,x,U)|\le\frac{(k-1)(k+5)U}{8x}(\log x)^{k-2}.
\end{equation}
This completes the proof.
\end{proof}
 
 \section{Proof of Theorem \ref{dfg}}
We extend the approach in Theorem \ref{ric} to $T_{k}(x)$, where understanding similar bounds for higher-order divisor functions becomes essential.
        \begin{proof} For the proof of Theorem \ref{dfg}, we used a simple modification of the proof by induction described in \cite[Theorem 12.1]{MR882550}. 
            
     Using the hyperbola method and the identity $d_k(n)=(1*d_{k-1})(n)$ yields 
        \begin{align}
            T_k(x)&=\sum_{an\le x}d_{k-1}(n)\notag\\
            &=\sum_{a\le U}\sum_{n\le x/a}d_{k-1}(n)+\sum_{U< a\le x}\sum_{n\le x/a}d_{k-1}(n)\notag\\
             &=\sum_{a\le U}T_{k-1}\left(\frac{x}{a}\right)+\sum_{n\le x/U}d_{k-1}(n)\sum_{U< a\le x/n}1\notag\\
             &=\sum_{a\le U}T_{k-1}\left(\frac{x}{a}\right)+x\sum_{n\leq x/U}\frac{d_{k-1}(n)}{n}-U\cdot T_{k-1}\left(\frac{x}{U}\right)+R_1(k,x,U),\label{cvg}
            \end{align}  
            where the error term, $R_{1}(k,x,U)=\vartheta(T_{k-1}(x/U))$, and is bounded above by
             \begin{align*}
            |R_1(k,x,U)|&\leq T_{k-1}(x/U)\\ &\le\frac{x}{U}P_{k-1}(\log(x/U))+\lambda_{k- 1}\left(\frac{x}{U}\right)^\frac{k-2}{k-1}(\log(x/U))^\frac{(k-2)(k-3)}{2(k-1)}.
        \end{align*}
Hence, it follows from \eqref{ggbb} that
\begin{equation}\label{jojo}
   |R_1(k,x,U)| <\frac{(k-1)x}{U}(\log x)^{k-2}+\lambda_{k- 1}\left(\frac{x}{U}\right)^\frac{k-2}{k-1}(\log(x/U))^\frac{(k-2)(k-3)}{2(k-1)}.
\end{equation}

        For the first and second sums on the right-hand side of equation \eqref{cvg}, we apply the results from Lemma \ref{ccx} and Lemma \ref{cmc}, respectively. For both sums, Lemma \ref{ccx} and Lemma \ref{cmc}  provide the main terms and an estimate of the error terms. In the context of the overall expression on the right-hand side of \eqref{cvg}, this enables us to assess and simplify the contributions of both sums.
        
        Lastly, we consider the third term in \eqref{cvg}. By applying \eqref{mlo} to this term, we obtain the following reformulation:
        \begin{align}\label{rfp}
            U\cdot T_{k-1}\left(\frac{x}{U}\right)&=U\left(\frac{xP_{k-1}(\log(x/U))}{U}+R_4(x,U)\right),
             \end{align}
where $R_4(k,x,U)= \Delta_{k-1}(x/U)$. Hence, for all $x/U> 1$ we have
\begin{align}\label{dev}
    |R_4(k,x,U)|&\leq \lambda_{k-1}\left(\frac{x}{U}\right)^\frac{k-2}{k-1}(\log x)^\frac{(k-2)(k-3)}{2(k-1)}.
\end{align}

Substituting \eqref{nd}, \eqref{fes}, and \eqref{rfp} into \eqref{cvg}, we obtain the main term in \eqref{mjj} as required with 
   \begin{align*}
       \Delta_{k}(x,U)&=R_{1}(k,x,U)+R_{2}(k,x,U)+xR_3(k,x,U)-UR_{4}(k,x,U).
   \end{align*}
   Now, the upper bound on $\Delta_k(x, U)$ is given by 
   \begin{align}\label{bbss}
       |\Delta_{k}(x,U)|&< |R_1(k,x,U)|+|R_2(k,x,U)|+x|R_3(k,x,U)|+U|R_4(k,x,U)|\notag\\
       &=\frac{3(k-1)x}{2U}(\log x)^{k-2}+\lambda_{k- 1}\left(\frac{x}{U}(\log(x/U))^\frac{k-3}{2}\right)^\frac{k-2}{k-1}\notag\\
    &\quad+k\lambda_{k-1}U^\frac{1}{k-1}(x(\log x)^{\frac{k-3}{2}})^{\frac{k-2}{k-1}}-(k-1)\lambda_{k-1}
        (x(\log x)^{\frac{k-3}{2}})^{\frac{k-2}{k-1}}\notag\\
    &\quad+\frac{(k-1)(k-2)x}{8U^{2}\log U} (\log x)^{k-2}+\frac{(k-1)(k+5)U}{8}(\log x)^{k-2}\notag\\
    &\quad+ \frac{(k^{3}-4k^{2}+7k-4)\lambda_{k-1}U^{\frac{1}{k-1}}(\log (x/U))^\frac{k^{2}-7k+8}{2(k-1)}}{2x^{\frac{1}{k-1}}}+\frac{k\lambda_{k-1}U^{\frac{1}{k-1}}}{x^{\frac{1}{k-1}}}\notag\cdot\\&\qquad{} (\log (x/U))^\frac{(k-2)(k-3)}{2(k-1)}.
       \end{align}
    
    Since $x>U>1$, the first and third terms on the right-hand side of \eqref{bbss} are asymptotically larger than the remaining terms.  Hence, we set
    \begin{align}\label{bvn}
        U&=cx^\frac{1}{k}(\log x)^\frac{(k-2)(k+1)}{2k},
    \end{align}
    where $c$ is a constant. Optimizing $U$ over $c$ and simplifying gives $c$ as in \eqref{ttr}.
     Also, we have 
   \begin{align*}
       \log\left(\frac{x}{U}\right)&=\log x\left(\frac{k-1}{k}-\frac{1}{\log x_{0}}\left(\log c+\frac{(k-2)(k+1)\log(\log x_{0})}{2k}\right)\right)\\
       &\le \log x,
   \end{align*}
for all, $k\ge 3$, $x\ge x_{0}$ and $0<c< 2$.

    Substituting \eqref{bvn} into \eqref{bbss} and simplifying yields 
\begin{align*}
|\Delta_{k}(x)|&< \left(\frac{3(k-1)}{2c}+k\lambda_{k-1}c^{\frac{1}{k-1}}\right)x^{\frac{k-1}{k}}(\log x)^{\frac{(k-2)(k-1)}{2k}}+\frac{(k-1)(k+5)c\,x^{1/k}}{8}\cdot\notag\\&\quad(\log x)^{\frac{(k-2)(3k+1)}{2k}}+\left(\frac{(k-1)(k-2)}{8c^{2}}+\frac{\lambda_{k-1}}{c^{\frac{k-2}{k-1}}}\right)\left(\frac{x}{\log x}\right)^{\frac{k-2}{k}}\notag\\&\quad+\lambda_{k-1}c^{\frac{1}{k-1}}\left(\frac{k^{3}}{2}+k\log x\right)\frac{(\log x)^{\frac{k^{2}-5k+2}{2k}}}{x^{1/k}}\\
&<\lambda_{k} x^{\frac{k-1}{k}}(\log x)^{\frac{(k-2)(k-1)}{2k}},
\end{align*}
where for all $x\ge x_{0}$, we have $\lambda_{k}$ and $c$ as in \eqref{cxz} and \eqref{ttr} respectively.
    
Finally, we obtain the results for $\lambda_{k}$ with $k=3$ in Remark~\ref{dmf}, and for $k=4, 5,$ and $6 $ in Table~\ref{lambda}, valid for all $x\ge x_{0}$.
\end{proof}

\section{Conclusion and Future Work}
Here, we briefly describe two theoretical approaches that may have contributed to the worse constant for $\lambda_{k}$ in Theorem \ref{dfg}. It is worth noting that the choice of the optimal constant in \eqref{abn} and \eqref{ttr} affects the resulting estimate. However, these constants were not arbitrarily chosen. For instance, a simple choice of $A=1$ in \eqref{abn} gives a better estimate for $\Delta_{3}(x)$ compared with $c=1.852$ in \eqref{ttr}. Also, to obtain a sharper bound for $\lambda_{k}$, one may consider an improvement in $\lambda_{k-1}$.  In contrast to \eqref{tk}, which was used in Theorem \ref{dfg}, we applied a better estimate for $\Delta(x)$ in Theorem \ref{ric}, as shown in \eqref{amo}, which holds for $x$ sufficiently large. Other estimates in the same range of $x$ necessitated this estimate. 

In this paper, we have improved upon the logarithmic exponent in the error term for $T_k(x)$ by using an elementary approach, demonstrating that meaningful refinement could be achieved without recourse to analytic techniques. However, we are excited to build upon this result using analytic methods,  with the aim of obtaining an explicit estimate for $\Delta_{k}(x)$ with an improvement in the exponent of the $x$ factor (see \cite[Theorems 12.2 and 12.3]{MR882550}). 
\section*{Acknowledgements}
The author would like to thank his supervisor, Timothy S. Trudgian, for his  assistance throughout this project. The author also acknowledges Bryce Kerr, Neea Paloj{\"a}rvi, Olivier Bordell\`es, Ali Ebadi,  
Michaela Cully-Hugill, and Hongliang Wang for their valuable discussions. 

\printbibliography
\end{document}